\documentclass[11pt]{amsart}

\usepackage{amssymb, amscd}
\usepackage{epsfig, mathtools}
\usepackage[vmargin=1in, hmargin=1.25in]{geometry}
\usepackage[font=small,format=plain,labelfont=bf,up,textfont=it,up]{caption}
\usepackage[shortlabels]{enumitem}
\usepackage[backref=page, bookmarks, bookmarksdepth=2, colorlinks=true, linkcolor=blue, citecolor=blue, urlcolor=blue]{hyperref}
\usepackage{etoolbox}
\usepackage{tikz-cd}
\apptocmd{\thebibliography}{\raggedright}{}{}

\makeatletter
\patchcmd{\@maketitle}{\global\topskip42\p@\relax}
  {\global\topskip42\p@\relax \vspace*{-38pt}}
  {}{}
\makeatother

\setenumerate[0]{label=(\alph*)}

\renewcommand*{\backref}[1]{}
\renewcommand*{\backrefalt}[4]{%
    \ifcase #1 (Not cited.)%
    \or        (Cited on page~#2.)%
    \else      (Cited on pages~#2.)%
    \fi}

\newcommand{\arxiv}[1]{\href{http://arxiv.org/abs/#1}{{\tt arXiv:#1}}}

\numberwithin{equation}{section}

\theoremstyle{plain}
\newtheorem{theorem}{Theorem}[section]
\newtheorem{maintheorem}{Theorem}

\newtheorem{proposition}[theorem]{Proposition}
\newtheorem{lemma}[theorem]{Lemma}

\newtheorem{conjecture}[theorem]{Conjecture}
\newtheorem{question}[theorem]{Question}
\newtheorem*{unnumberedquestion}{Question}

\theoremstyle{definition}
\newtheorem{asm}[theorem]{Assumption}

\newtheorem{defn}[theorem]{Definition}

\newtheorem{notn}[theorem]{Notation}

\theoremstyle{remark}
\newtheorem{rmk}[theorem]{Remark}
\newenvironment{remark}[1][]{\begin{rmk}[#1] \pushQED{\qed}}{\popQED \end{rmk}}
\newtheorem{eg}[theorem]{Example}
\newenvironment{example}[1][]{\begin{eg}[#1] \pushQED{\qed}}{\popQED \end{eg}}

\DeclareMathOperator{\SU}{SU}

\newcommand\C{\ensuremath{\mathbb{C}}}
\newcommand\Z{\ensuremath{\mathbb{Z}}}

\newcommand\N{\ensuremath{\mathbb{N}}}

\DeclareMathOperator{\HH}{H}

\newcommand\Set[2]{\ensuremath{\left\{\text{#1 $|$ #2}\right\}}}
\newcommand\GroupPres[2]{\ensuremath{\left\langle \text{#1 $|$ #2} \right\rangle}}

\newcommand\cI{\ensuremath{\mathcal{I}}}

\newcommand\cL{\ensuremath{\mathcal{L}}}

\newcommand\cP{\ensuremath{\mathcal{P}}}

\newcommand\cS{\ensuremath{\mathcal{S}}}

\newcommand\bm{\ensuremath{\mathbf{m}}}

\newcommand\bs{\ensuremath{\mathbf{s}}}
\newcommand\bt{\ensuremath{\mathbf{t}}}

\newcommand\bx{\ensuremath{\mathbf{x}}}
\newcommand\by{\ensuremath{\mathbf{y}}}

\title[Word length vs lower central series depth for surface groups and RAAGs]{Word length versus lower central series depth for surface groups and RAAGs}

\author{Justin Malestein}
\address{Dept of Mathematics; University of Oklahoma; 601 Elm Ave; Norman, OK 73091}
\email{jmalestein@ou.edu}

\author{Andrew Putman}
\address{Dept of Mathematics; University of Notre Dame; 255 Hurley Hall; Notre Dame, IN 46556}
\email{andyp@nd.edu}

\thanks{JM was supported in part by a Simons Foundation Collaboration Grant 713006.  AP was supported in
part by NSF Grant DMS-2305183.}

\date{January 18, 2024}

\begin{document}

\newpage

\begin{abstract}
For surface groups and right-angled Artin groups, we prove lower bounds on the shortest word in the generators
representing a nontrivial element of the $k^{\text{th}}$ term of the lower central series.
\end{abstract}

\maketitle
\thispagestyle{empty}

\section{Introduction}

Let $G$ be a group and let $\gamma_{k}(G)$ be its lower central series:
\[\gamma_1(G) = G \quad \text{and} \quad \gamma_{k+1}(G) = [\gamma_k(G),G] \quad \text{for $k \geq 1$}.\]
If $\gamma_{k+1}(G) = 1$, then $G$ is at most $k$-step nilpotent.  Let $S$ be a finite generating set
for $G$. 

\begin{unnumberedquestion}
What is the shortest word in $S^{\pm 1}$ representing
a nontrivial element in $\gamma_k(G)$?  What are the asymptotics of the length of this word
as $k \to \infty$?
\end{unnumberedquestion}

The asymptotic question is only interesting for non-nilpotent groups.  It is also natural
to only consider groups that are residually nilpotent, i.e., such that
\[\bigcap_{k=1}^{\infty} \gamma_k(G) = 1\]
Let $G$ be a non-nilpotent residually nilpotent group with a finite generating set $S$.  Define for $g \in G$ its associated word norm:
\[\|g\|_S = \min \Set{$\ell$}{$g$ can be written as a word of length $\ell$ in $S^{\pm 1}$}.\]
The {\em lower central series depth function}
is the following function 
$d_{G,S}\colon \N \rightarrow \N$:
\[d_{G,S}(k) = \min \Set{$\|g\|_S$}{$g \in \gamma_k(G)$, $g \neq 1$}.\]
Though $d_{G,S}(k)$ depends on the generating set $S$, its asymptotic behavior as $k \to \infty$ is independent
of $S$.  Our goal in this paper is to find bounds on $d_{G,S}(k)$ for
several natural classes of groups $G$.

\subsection{Free groups}
For $n \geq 2$, let $F_n$ be the free group on
$S = \{x_1,\ldots,x_n\}$.  These are the most fundamental examples
of groups that are residually nilpotent but not nilpotent \cite{Magnus}, and both lower and upper bounds on $d_{F_n,S}(k)$ have been studied:
\begin{itemize}
\item Using the free differential calculus, Fox \cite[Lemma 4.2]{FoxFree1} proved that
$d_{F_n,S}(k) \geq \frac{1}{2} k$ for $k \geq 1$.
In \cite[Theorem 1.2]{MalesteinPutmanFree}, the authors improved this to $d_{F_n,S}(k) \geq k$.
\item In \cite[proofs of Theorems 1.3 and 1.5]{MalesteinPutmanFree}, the authors proved that
$d_{F_n,S}(k) \leq \frac{1}{4}(k+1)^2$.  Elkasapy--Thom \cite[Theorem 2.2]{ElkasapyThom} then
improved this to a bound that grows like $k^c$ with $c \approx 1.4411$, and later Elkasapy \cite[Theorem 1.1]{Elkasapy} slightly
improved this to $k^c$ where $c = \log_\varphi 2$ and $\varphi$ is the golden ratio.
\end{itemize}
The growth rate of $d_{F_n,S}(k)$ thus lies between $k$ and $k^{\log_\varphi 2}$, and Elkasapy conjectured
that the growth rate is $k^{\log_\varphi 2}$.  

\begin{remark}
Kuperberg \cite{Kuperberg} showed that finding short words that are deep in $\gamma_k(F_n)$ is related
to the problem of approximating elements of $\SU(d)$ by elements of a
dense subgroup.  He showed you can use such short elements of $\gamma_k(F_n)$ in
an algorithm that solves the following problem: given a finitely generated dense subgroup $\Gamma$ of $\SU(d)$, a
specified error tolerance, and a specific element $M \in \SU(d)$, construct short words in $\Gamma$
that approximate $M$ to the given tolerance.
See \cite{Kuperberg} for precise results and more context.
\end{remark}

\subsection{Upper bounds}
\label{subsection:upperbounds}

Now let $G$ be a non-nilpotent residually nilpotent group with a finite generating set $S$.  If
$G$ contains a non-abelian free subgroup, then using the work of Elkasapy--Thom discussed above
we can find an upper bound on $d_{G,S}(k)$ that grows\footnote{Precise upper bounds are
more complicated and depend on how the free subgroup is embedded in $G$.} like $k^{1.4411}$.  However, lower bounds
on $d_{G,S}(k)$ do not follow from the analogous results for free groups, so for
the rest of this paper we focus on lower bounds.

\subsection{Surface groups}
Let $\Sigma_g$ be a closed oriented genus $g \geq 2$ surface and let
\[\pi = \pi_1(\Sigma_g) = \GroupPres{$a_1,b_1,\ldots,a_g,b_g$}{$[a_1,b_1] \cdots [a_g,b_g] = 1$}.\]
Here our convention is that $[x, y] = xyx^{-1}y^{-1}$.
The surface group $\pi$ is residually nilpotent but not nilpotent \cite{Baumslag, Frederick}, and shares many features
with free groups.  Since $g \geq 2$, the subgroup of $\pi$ generated by $a_1$ and $b_1$ is a rank $2$
free group.  As in \S \ref{subsection:upperbounds} above, this implies a
$k^{1.4411}$ upper bound on the growth rate of $d_{\pi,S}(k)$.

However, lower bounds are more problematic.  The known lower bounds for free groups use
the free differential calculus, and there is no analogue of the free differential
calculus for surface groups.\footnote{The free derivatives are derivations
$d\colon F_n \rightarrow \Z[F_n]$.  For a group $G$, if there exist nontrivial derivations
$d\colon G \rightarrow \Z[G]$ then $\HH^1(G;\Z[G]) \neq 0$.  If $G$ has a compact $K(G,1)$ this implies that $G$ has more
than one end \cite{ScottWallTopological}, so $G$ cannot be a one-ended group like a surface group.}
The lower bounds for free groups can also be derived using
the ``Magnus representations'' from free groups to units in rings of power series with
noncommuting variables,
but again it seems hard to construct suitable analogues for surface groups.
Nevertheless, we are able to prove the following:

\begin{maintheorem}
\label{maintheorem:surfacegroups}
Let $\pi$ be a nonabelian surface group with standard generating set
$S = \{a_1,b_1,\ldots,a_g,b_g\}$.  Then for all $k \geq 1$ we have
$d_{\pi,S}(k) \geq \frac{1}{4} k$.
\end{maintheorem}

The $\frac{1}{4}$ in this theorem is probably not optimal.  We make the following conjecture:

\begin{conjecture}
Let $\pi$ be a nonabelian surface group with standard generating set
$S = \{a_1,b_1,\ldots,a_g,b_g\}$.  Then $d_{\pi,S}(k) \geq k$ for all $k \geq 1$.
\end{conjecture}

See \S \ref{section:optimalbounds} below for why our proof likely cannot be extended to prove
this conjecture.

\subsection{Right-angled Artin groups}
We will derive Theorem \ref{maintheorem:surfacegroups} from an analogous result
for right-angled Artin groups, which are defined as follows.
Let $X$ be a finite graph.  The associated {\em right-angled Artin group} (RAAG) is the group $A_X$ given
by the following presentation:
\begin{itemize}
\item The generators are the vertex set $V(X)$.
\item The relations are $\Set{$[x,y]=1$}{$x,y \in V(X)$ are joined by an edge}$.
\end{itemize}

\begin{example}
The free abelian group $\Z^n$ is the RAAG with $X$ the complete graph on $n$ vertices, and the free group $F_n$ is the RAAG
with $X$ a graph with $n$ vertices and no edges.
\end{example}

These groups play an important role in many areas of geometric group theory (see, e.g., \cite{CharneySurvey, WiseBook}).
Just like free groups and surface groups, they are residually nilpotent \cite{DromsThesis}, and
they are only nilpotent if they are free abelian, i.e., if $X$ is a complete graph.
The latter fact can be deduced from the basic observation that if $Y$ is a vertex-induced
subgraph of $X$, then the natural map $A_Y \to A_X$ is split injective; indeed, 
the map $A_X \to A_Y$ that kills the generators which are not vertices of $Y$ is a right inverse
for it.

\begin{remark}
\label{remark:baudisch}
More generally, Baudisch \cite{Baudisch} proved that any two elements of a RAAG either commute or generate a free
subgroup.  This implies that any nonabelian subgroup of a RAAG is non-nilpotent.
\end{remark}

Right-angled Artin groups 
often contain many surface subgroups \cite{CrispSageevSapir, CrispWiest, KimSurface, ServatiusDromsServatius}, and
we will prove Theorem \ref{maintheorem:surfacegroups} by embedding surface groups into RAAGs
and studying the lower central series depth function there.  The main result 
we need along these lines is as follows.  

\begin{maintheorem}
\label{maintheorem:raag}
Let $X$ be a finite graph that is not a complete graph,
and let $S = V(X)$ be the generating set of $A_X$.  Then
for $k \geq 1$ we have $d_{A_X,S}(k) \geq k$.
\end{maintheorem}

Though Theorem \ref{maintheorem:raag} does not seem to previously appear
in the literature, it is implicit in the work of Wade (see \cite[Lemma 4.7]{WadeSurvey}), and
our proof follows his ideas.
The key tool is 
a version of the ``Magnus representation'' for RAAGs 
that was introduced by Droms in his thesis \cite{DromsThesis}, generalizing
work of Magnus on free groups.  The classical Magnus representations
are maps from $F_n$ to units in rings of power series with noncommuting variables (see \cite[Chapter 5]{MagnusKarrassSolitar}).
They contain much of the same information as the free derivatives.

\subsection{From RAAGs to surface groups}
\label{section:fromraags}

Let $G$ be a non-nilpotent residually nilpotent group with finite generating set $T$ and let $H$ be the subgroup of $G$ generated
by a finite subset $S<G$.  Each $s \in S$ can be written as a word in $T^{\pm 1}$, so we can define
\[r = \max\Set{$\|s\|_T$}{$s \in S$}.\]
For $h \in H$, we thus have
\[\|h\|_S \geq \frac{1}{r} \|h\|_T.\]
From this, we see that
\[d_{H,S}(k) \geq \frac{1}{r} d_{G,S}(k) \quad \text{for all $k \geq 1$}.\]
Since all nonabelian surface groups $\pi$ are subgroups of RAAGs, Theorem \ref{maintheorem:raag} therefore immediately
implies a linear lower bound on the lower central series depth function of $\pi$. 
However, the precise constants depend on the embedding into a RAAG, and without further
work might depend on the genus $g$.
To get the genus-independent constant $\frac{1}{4}$ from Theorem \ref{maintheorem:surfacegroups}, 
we will have to carefully control the geometry of our embeddings of surface groups into RAAGs
and ensure that we can take $r=4$ in the above.

\begin{remark}
\label{remark:lowerbound}
Many other groups can also be embedded in right-angled Artin groups, and the argument
above shows that all of them have linear lower bounds on their lower central series
depth functions (which are well-defined by Remark \ref{remark:baudisch}).
\end{remark}

\subsection{Optimal embeddings}
\label{section:optimalbounds}

It is natural to wonder if we can improve the $\frac{1}{4}$ in Theorem \ref{maintheorem:surfacegroups}
by using a more clever embedding into a RAAG.  We conjecture that this is not possible:

\begin{conjecture}
\label{conjecture:optimalraag}
Let $\pi$ be a nonabelian surface group with standard generating set
$S = \{a_1,b_1,\ldots,a_g,b_g\}$, let $X$ be a finite graph, and let
$\phi\colon \pi \hookrightarrow A_X$ be an embedding.  Then there exists
some $s \in S$ such that $\|\phi(s)\|_{V(X)} \geq 4$.
\end{conjecture}

\begin{remark}
As we will discuss in \S \ref{section:surfaceraags} below, Crisp--Wiest \cite{CrispWiest} gave
an explicit description of all homomorphisms from surface groups to RAAGs in terms of collections of
loops on the surface.  To prove Conjecture \ref{conjecture:optimalraag}, what one would have
to show is that if $\phi\colon \pi \rightarrow A_X$ is a map from a surface group to a RAAG
arising from the Crisp--Wiest construction that does not satisfy the conclusion of
Conjecture \ref{conjecture:optimalraag}, then $\phi$ is not injective.
\end{remark}

\subsection{Sublinearity}

We close by posing the following question:

\begin{question}
Does there exist a non-nilpotent residually nilpotent group $G$ equipped with a finite generating set $S$
such that $d_{G,S}$ grows sublinearly?
\end{question}

By Remark \ref{remark:lowerbound}, such a group $G$ cannot be a subgroup of a RAAG.

\subsection{Outline}
We prove Theorem \ref{maintheorem:raag} in \S \ref{section:raags} and Theorem \ref{maintheorem:surfacegroups} 
in \S \ref{section:surfacegroups}.  This last
section depends on the preliminary
\S \ref{section:surfaceraags}, which discusses work of Crisp--Wiest parameterizing maps from surface groups to RAAGs.

\subsection{Acknowledgments}
We thank Greg Kuperberg for some useful references.

\section{Right-angled Artin groups}
\label{section:raags}

Let $X$ be a finite graph with associated right-angled Artin group $A_X$.  In this section,
we first discuss some structural results about $A_X$ and then prove Theorem \ref{maintheorem:raag}.

\subsection{Monoid}
In addition to the right-angled Artin group $A_X$, we will also need the right-angled Artin monoid $M_X$.  This
is the associative monoid with the following presentation:
\begin{itemize}
\item The generators are the vertices $V(X)$ of $X$.  To distinguish these generators from the
corresponding generators of $A_X$, we will sometimes write them with bold-face letters.  In other words, $s$ denotes
an element of $A_X$ and $\bs$ denotes an element of $M_X$.
\item The relations are $\Set{$\bx\by = \by\bx$}{$x,y \in V(X)$ are joined by an edge}$.
\end{itemize}
There is a monoid homomorphism $M_X \rightarrow A_X$ whose image is the set of all elements
of $A_X$ that can be represented by ``positive words''.  As we will discuss below, this
monoid homomorphism is injective.

\subsection{Normal form}
Let $S = V(X)$ be the generating set for $A_X$ and $M_X$.  Consider a word
\[w = s_1^{e_1} \cdots s_n^{e_n} \quad \text{with $s_1,\ldots,s_n \in S$ and $e_1,\ldots,e_n \in \Z$}.\]
This word represents an element of $A_X$, and if $e_i \geq 0$ for all $1 \leq i \leq n$ it represents
an element of $M_X$ (here for conciseness we are not using our bold-face conventions).  We say that $w$ is {\em fully reduced} if it satisfies the following conditions:
\begin{itemize}
\item Each $e_i$ is nonzero.
\item For all $1 \leq i < j \leq n$ with $s_i = s_j$, there exists some $k$ with $i < k < j$ such that
$s_k$ does not commute\footnote{As observed earlier, $A_Y$ embeds in $A_X$ for any
vertex-induced subgraph $Y$, so this is equivalent to $s_k$ being distinct from
and not adjacent to $s_i = s_j$.} with $s_i = s_j$.
\end{itemize}
Note that this implies in particular that $s_i \neq s_{i+1}$ for all $1 \leq i < n$, so $w$ is reduced as
a word in the free group on $S$.  It is clear that every element of $A_X$ and $M_X$ can be represented
by a fully reduced word.  

This representation is unique in the following sense:
\begin{itemize}
\item Consider fully reduced words
\[w = s_1^{e_1} \cdots s_n^{e_n} \quad \text{and} \quad w' = t_1^{f_1} \cdots t_{m}^{f_m}\]
representing the same element of $A_X$ or $M_X$.  Then we can obtain $w'$ from $w$ by a sequence of {\em swaps}, i.e.,
flipping adjacent terms $s_i^{e_{i}}$ and $s_{i+1}^{e_{i+1}}$ such that $s_i$ commutes with $s_{i+1}$.
\end{itemize}
For $A_X$, this uniqueness was stated without proof by Servatius \cite{ServatiusAutos}.  The earliest proof
we are aware of is in Green's thesis \cite{GreenThesis}.  Alternate proofs can be found in
\cite[Proposition 9]{CrispWiest} and \cite[Theorem 4.14]{WadeSurvey}.
Using the monoid homomorphism $M_X \rightarrow A_X$, the uniqueness for $M_X$ follows\footnote{Whether this is a
circular argument depends on the proof of uniqueness used for $A_X$.  The geometric proof
from \cite[Proposition 9]{CrispWiest} works directly with groups, and does not even implicitly
prove anything about monoids.} from that of $A_X$.  Note that this uniqueness also implies
that the monoid homomorphism $M_X \rightarrow A_X$ is injective.

The following lemma shows that fully reduced words realize the word norm in $A_X$:

\begin{lemma}
\label{lemma:fullyshort}
Let $X$ be a finite graph.  Let $S = V(X)$ be the generating set for $A_X$.  Consider some $w \in A_X$, and represent
$w$ by a fully reduced word
\[w = s_1^{e_1} \cdots s_n^{e_n} \quad \text{with $s_1,\ldots,s_n \in S$ and $e_1,\ldots,e_n \in \Z$}.\]
Then $\|w\|_S = |e_1|+\cdots+|e_n|$.
\end{lemma}
\begin{proof}
Immediate from the uniqueness up to swaps of fully reduced words as well as the fact that taking
an arbitrary word and putting it in fully reduced form does not lengthen the word.
\end{proof}

\subsection{Monoid ring}

Let $\Z[M_X]$ be the monoid ring whose elements are formal $\Z$-linear combinations of elements of $M_X$.  
Since the relations in $M_X$ are all of the form $\bx\by = \by\bx$ for generators $\bx$ and $\by$, 
all words representing an element $\bm \in M_X$ have the same length, which we will denote $\ell(\bm)$.  This length function
satisfies $\ell(\bm_1 \bm_2) = \ell(\bm_1)+\ell(\bm_2)$ for $\bm_1,\bm_2 \in M_X$.  For $k \geq 0$, define
\[M_X^{(k)} = \Set{$\bm \in M_X$}{$\ell(\bm) = k$}.\]
The monoid ring $\Z[M_X]$ is a graded ring with $\Z[M_X]_{(k)} = \Z[M_X^{(k)}]$.

\subsection{Partially commuting power series}

Let $I \subset \Z[M_X]$ be the ideal generated by the elements of the generating set $V(X)$.  For $k \geq 1$,
the ideal $I^k$ consists of $\Z$-linear combinations of $\bm \in M_X$ with $\ell(\bm) \geq k$.  Define
\[\cP_X = \lim_{\longleftarrow} \Z[M_X]/I^k.\]
Elements of the inverse limit $\cP_X$ can be regarded as power series
\[\sum_{k=0}^{\infty} \bm_k \quad \text{with $\bm_k \in \Z[M_X]_{(k)}$ for all $k \geq 0$}.\]
Each $\bm_k$ is a linear combination of products of $k$ generators from $V(X)$, some of which commute and
some of which do not.  Multiplication works in the usual way:
\[\left(\sum_{k=0}^{\infty} \bm_k\right) \left(\sum_{k'=0}^{\infty} \bm'_{k'}\right) = \sum_{\ell=0}^{\infty} \left(\sum_{k+k' = \ell} \bm_k \bm'_{k'}\right).\]

\subsection{Magnus representation}

We now discuss the Magnus representation of $A_X$, which was introduced by Droms in his thesis \cite{DromsThesis}, generalizing
classical work of Magnus for free groups (see \cite[Chapter 5]{MagnusKarrassSolitar}).  See \cite{WadeSurvey} for
a survey.  The starting point is the observation that for $s \in V(X)$, we have the following identity in
$\cP_X$:
\[(1+\bs)(1-\bs+\bs^2-\bs^3+\cdots) = 1.\]
In other words, $1+\bs$ is a unit in $\cP_X$.  If generators $s,s' \in V(X)$ commute, then $1+\bs$ and $1+\bs'$ also commute.  It follows that we can define
a homomorphism
\[\mu\colon A_X \longrightarrow \left(\cP_X\right)^{\times}\]
via the formula
\[\mu(s) = 1+\bs \quad \text{for $s \in V(X)$}.\]

\subsection{Dimension subgroups and the lower central series}

Recall that $I \subset \Z[M_X]$ is the ideal generated by elements of the generating set $V(X)$.  There
is a corresponding ideal $\cI \subset \cP_X$ consisting of 
all elements with constant term $0$. 
For $k \geq 1$, the $k^{\text{th}}$ {\em dimension subgroup} of $A_X$,
denoted $D_k(A_X)$, is the kernel of the composition
\[A_X \stackrel{\mu}{\longrightarrow} \cP_X \longrightarrow \cP_X/\cI^k.\]
In other words, $D_k(A_X)$ consists of elements $w \in A_X$ such that
\[\mu(w) = 1 + \left(\text{terms of degree at least $k$}\right).\]
The most important theorem about $D_k(A_X)$ identifies it with the $k^{\text{th}}$ term
of the lower central series of $A_X$:

\begin{theorem}[{\cite[Theorem 6.3]{WadeSurvey}}]
\label{theorem:lcsartin}
Let $X$ be a finite graph.  Then $D_k(A_X) = \gamma_k(A_X)$ for all $k \geq 1$.
\end{theorem}

\begin{remark}
In fact, for what follows all we need is the much easier fact that $\gamma_k(A_X) \subset D_k(A_X)$,
which appears in Droms's thesis \cite{DromsThesis}.  For this, since $D_1(A_X) = A_X = \gamma_1(A_X)$
it is enough to verify that
\[[D_k(A_X),D_{\ell}(A_X)] \subset D_{k+\ell}(A_X),\]
which is immediate from the definitions.
\end{remark}

\subsection{Lower bounds for the lower central series of a RAAG}

We close this section by proving Theorem \ref{maintheorem:raag}.  As we said in the introduction,
the proof closely follows ideas of Wade \cite{WadeSurvey}.

\begin{proof}[Proof of Theorem \ref{maintheorem:raag}]
We start by recalling the statement.  Let $X$ be a finite graph that is not a complete graph and
let $S = V(X)$ be the generating set for $A_X$.  Consider a nontrivial element $w \in A_X$, and
let $k = \|w\|_S$ be its word norm in the generating set $S$.  We must prove that
$w \notin \gamma_{k+1}(A_X)$.  By Theorem \ref{theorem:lcsartin}, it is enough to prove
that $w \notin D_{k+1}(A_X)$.

Represent $w$ by a fully reduced word:
\[w = s_1^{e_1} \cdots s_n^{e_n} \quad \text{with $s_1,\ldots,s_n \in S$ and $e_1,\ldots,e_n \in \Z$}.\]
By Lemma \ref{lemma:fullyshort}, we have
\[\|w\|_S = |e_1|+\cdots+|e_n| \geq n.\]
It is thus enough to prove that $w \notin D_{n+1}(A_X)$.  To do this, is enough to 
prove that a term of degree $n$ appears in $\mu(w) \in \cP_X$.

An easy induction shows that for all $1 \leq i \leq n$, we have
\[\mu(s_i^{e_i}) = (1+\bs_i)^{e_i} = 1 + e_i \bs_i + \bs_i^2 \bt_i \quad \text{for some $\bt_i \in \cP_X$}.\]
It follows that
\begin{equation}
\label{eqn:expandmu}
\mu(w) = \left(1+e_1 \bs_1 + \bs_1^2 \bt_1\right) \left(1+e_2 \bs_2 + \bs_2^2 \bt_2\right) \cdots \left(1+e_n \bs_n + \bs_n^2 \bt_n\right).
\end{equation}
Say that some $\bm \in M_X$ is {\em square-free} if it cannot be expressed as a word
in the generators $S = V(X)$ for the monoid $M_X$ with two consecutive letters the same generator.\footnote{Be warned
that it is possible for an element to have one such expression while not being square-free.
For instance, if $\bs,\bs' \in S$ are distinct commuting generators then $\bs \bs' \bs$ is not square-free
since $\bs \bs' \bs = \bs^2 \bs'$.}  It is immediate from the uniqueness up to swaps of fully reduced
words that the fully reduced word $\bs_1 \bs_2 \cdots \bs_n$ represents a square-free element of
$M_X$.  When we expand out \eqref{eqn:expandmu}, the only square-free term of degree
$n$ is
\[e_1 e_2 \cdots e_n \bs_1 \bs_2 \cdots \bs_n.\]
It follows that this degree $n$ term survives when we expand out $\mu(w)$, as desired.
\end{proof}

\section{Mapping surface groups to RAAGs}
\label{section:surfaceraags}

Before we can prove Theorem \ref{maintheorem:surfacegroups}, we must
discuss some work of Crisp--Wiest \cite{CrispWiest} that parameterizes maps from surface groups to RAAGs.
We will not need the most general form of their construction (which they prove can give {\em any} homomorphism
from a surface group to a RAAG), so we will only describe a special case of it.  Fix a closed oriented
surface $\Sigma$ and a basepoint $\ast \in \Sigma$.

\subsection{Crisp--Wiest construction}
A {\em simple dissection}\footnote{Crisp and Wiest use the term dissection for a collection of curves which satisfy some conditions and have a certain decoration. We add ``simple'' to indicate that
we do not have any decoration.} on $\Sigma$ is a finite collection $\cL$ of oriented simple closed curves
on $\Sigma$ satisfying the following conditions:
\begin{itemize}
\item None of the curves contain the basepoint $\ast$.
\item Any two curves in $\cL$ intersect transversely.
\item There are no triple intersection points between three curves in $\cL$.
\end{itemize}
For a simple dissection $\cL$, let $X(\cL)$ be the graph whose vertices are the curves in $\cL$ and where
two vertices are joined by an edge if the corresponding curves intersect.  Crisp--Wiest \cite{CrispWiest} proved that
the following gives a well-defined homomorphism $\phi\colon \pi_1(\Sigma,\ast) \rightarrow A_{X(\cL)}$:
\begin{itemize}
\item Consider some $x \in \pi_1(\Sigma,\ast)$.  Realize $x$ by an immersed based loop $\bx\colon [0,1] \rightarrow \Sigma$
that is transverse to all the curves in $\cL$ and avoids intersection points between curves of $\cL$.  If $\bx$ is disjoint from all the curves in $\cL$, then $\phi(x) = 1$.  Otherwise,
let
\[0 < t_1 < \cdots < t_n < 1\]
be the collection of all values such that $\bx(t_i)$ is contained in some $\gamma_i \in \cL$.  For $1 \leq i \leq n$, let
$e_i = \pm 1$ be the sign of the intersection of $\bx$ with the oriented loop $\gamma_i$ at $x(t_i)$.  Then
\[\phi(x) = \gamma_1^{e_1} \cdots \gamma_n^{e_n} \in A_{X(\cL)}.\]
\end{itemize}
We will say that $\phi$ is the map obtained by applying the {\em Crisp--Wiest construction} to $\cL$.

\subsection{Injectivity criterion}
\label{section:injectivitycriterion}

Crisp--Wiest \cite{CrispWiest} describe an approach for proving that $\phi$ is injective in certain cases.
To describe it, we must introduce some more terminology.  For a simple dissection
$\cL$ on $\Sigma$, let
\[G(\cL) = \bigcup_{\gamma \in \cL} \gamma,\]
which we view as a graph embedded in $\Sigma_g$ with a vertex for each intersection point
between curves in $\cL$.  We say that $\cL$ is a {\em filling curve system} if each
component of $\Sigma \setminus G(\cL)$ is a disk.  

For a component $U$ of $\Sigma \setminus G(\cL)$, the boundary of $U$ can be identified
with a circuit in the graph $G(\cL)$.  Say that $U$ satisfies the {\em injectivity criterion}
if the following holds for any two distinct edges $e$ and $e'$
in the boundary of $U$.  Let $\gamma$ and $\gamma'$ be the oriented curves in $\cL$
that contain $e$ and $e'$, respectively.  We then require that $\gamma \neq \gamma'$ and
that if $\gamma$ intersects $\gamma'$, then $e$ and $e'$ 
are adjacent edges in the boundary of $U$.

We can now state our injectivity criterion:

\begin{proposition}
\label{proposition:crispwiest}
Let $\Sigma$ be a closed oriented surface equipped with a basepoint
$\ast$ and let $\cL$ be a filling simple dissection on $\Sigma$.  
For all components $U$ of $\Sigma \setminus G(\cL)$, assume that $U$
satisfies the injectivity criterion.
Then the map $\phi\colon \pi_1(\Sigma,\ast) \rightarrow A_{X(\cL)}$ obtained
by applying the Crisp--Wiest construction to $\cL$ is injective.
\end{proposition}

While Proposition \ref{proposition:crispwiest} is not explicitly stated or proved in \cite{CrispWiest},
it is implicit in their work. We present a proof for the convenience of the reader. This requires some
preliminaries.

\begin{remark}
The injectivity criterion implies that the dual cubulation to $\cL$ we introduce below
is a {\em special cube complex} in the sense of Haglund--Wise \cite{HaglundWise}.  The paper
\cite{CrispWiest} predates \cite{HaglundWise}, and the machinery of \cite{HaglundWise} is
unnecessary for this application.
\end{remark}

\subsection{Salvetti complex}

Let $X$ be a finite graph and let $A_X$ be the corresponding right-angled Artin group.  The
{\em Salvetti complex} of $A_X$, denoted $\cS(X)$, is a certain non-positively curved cube complex\footnote{Here
a cube complex is non-positively curved if its universal cover is CAT(0).} with $\pi_1(\cS(X)) = A_X$.  
It can be constructed as follows.  Enumerate
the vertices of $X$ as
\[V(X) = \{v_1,\ldots,v_n\}.\]
Identify $S^1$ with the the unit circle in $\C$, so $1 \in S^1$ is a basepoint.  
For a subset $I \subset \{v_1,\ldots,v_n\}$ of cardinality $k$, let $S_I \cong (S^1)^k$ be 
\[S_I = \Set{$(z_1,\ldots,z_n) \in (S^1)^n$}{$z_i = 1$ for all $i$ with $v_i \notin I$}.\]
A subset $I \subset \{v_1,\ldots,v_n\}$ is a $k$-clique of $X$ if the subgraph of $X$
induced by $I$ is a complete subgraph on $k$ vertices.  A clique is a set of vertices
that forms a $k$-clique for some $k$.  With these definitions,
$\cS(X)$ is the union of the $S_I$ as $I$ ranges over cliques in $X$.  The space $\cS(X)$
can be given a cube complex structure containing a $k$-cube for each $k$-clique
in $X$.  In particular, it has a single vertex (i.e., $0$-cube) corresponding to the (empty)
$0$-clique.

\subsection{Dual cubulation}
Now let $\cL$ be a filling dissection on $\Sigma_g$.
We can form a cube complex structure on $\Sigma_g$ called the {\em cube complex structure dual to $\cL$} in the following way.
We start by defining the $1$-skeleton $\sigma$ of our cube complex structure:
\begin{itemize}
\item Put a vertex of $\sigma$ in the interior of each component of $\Sigma_g \setminus G(\cL)$.  For the component
containing the basepoint $\ast$, the vertex should be $\ast$.
\item For each edge $e$ of $G(\cL)$, connect the vertices in the components on either side of
$e$ by an edge of $\sigma$.
\end{itemize}
A component of $\Sigma_g \setminus G(\cL)$ looks like the following:\\
\centerline{\psfig{file=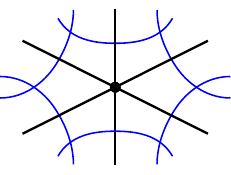,scale=1}}
Here $G(\cL)$ is in blue and $\sigma$ is in black.  Each edge coming out of the vertex of $\sigma$
shown in the figure terminates in the vertex in the adjacent component.

Now consider a component $C$ of $\Sigma_g \setminus \sigma$.  The component $C$ contains exactly one
vertex of $G(\cL)$, and the boundary of $C$ is composed of four edges of $\sigma$ as follows:\\
\centerline{\psfig{file=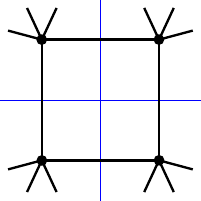,scale=1}}
Here again the graph $G(\cL)$ is blue and $\sigma$ is black.  Complete $\sigma$ to a cube complex
structure by attaching a square to each such $C$.

\subsection{Proof of Proposition \ref{proposition:crispwiest}}

We first recall what we must prove.  Let $\Sigma$ be a closed oriented surface equipped with a basepoint
$\ast$ and let $\cL$ be a filling simple dissection on $\Sigma$.
For all components $U$ of $\Sigma \setminus G(\cL)$, assume that $U$
satisfies the injectivity criterion.
We must prove that the map $\phi\colon \pi_1(\Sigma,\ast) \rightarrow A_{X(\cL)}$ obtained
by applying the Crisp--Wiest construction to $\cL$ is injective.

Endow $\Sigma$ with the cube complex structure dual to $\cL$, and let $\cS(X(\cL))$ be the Salvetti complex
of $A_{\cL(X)}$.  We start by constructing a map of cube complexes $f\colon \Sigma \rightarrow \cS(X(\cL))$
such that
\[f_{\ast}\colon \pi_1(\Sigma,\ast) \rightarrow \pi_1(\cS(X(\cL))) = A_{X(\cL)}\]
equals $\phi$.  Define $f$ as follows:
\begin{itemize}
\item The map $f$ sends each vertex of $\Sigma$ to the unique vertex of $\cS(X(\cL))$.
\item For an edge $e$ of $\Sigma$ that crosses an oriented loop $\gamma$ of $\cL$, the map
$f$ takes $e$ isometrically to the loop of $\cS(X(\cL))$ corresponding to the $1$-clique $\{\gamma\}$
of $X(\cL)$.  Orienting $e$ such that the intersection of $e$ with $\gamma$ is positive, we do
this such that $f(e)$ goes around the loop in the direction corresponding to the generator $\gamma$
of $\pi_1(\cS(X(\cL))) = A_{X(\cL)}$.
\item For a $2$-cube $c$ of $\Sigma$ centered at an intersection of loops $\gamma_1$ and $\gamma_2$
of $\cL$, the map $f$ sends $c$ isometrically to the $2$-cube corresponding to the $2$-clique
$\{\gamma_1,\gamma_2\}$ of $X(\cL)$.
\end{itemize}
With these definitions, it is clear that $f_{\ast} = \phi$.

By \cite[Theorem 1]{CrispWiest}, the map $f_{\ast} = \phi$ will be an injection if for every vertex
$v$ of $\Sigma$, the map $f$ take the link of $v$ injectively into a full subcomplex
of the link of $f(v)$ in $\cS(X(\cL))$.  These links have the following description:
\begin{itemize}
\item The vertex $v$ lies in some component $U$ of $\Sigma \setminus G(\cL)$.  The link of $v$ is
a cycle whose vertices are precisely the edges of $G(\cL)$ surrounding $U$.
\item The vertex $f(v)$ is the unique vertex of $\cS(X(\cL))$.  Its link is the following complex:
\begin{itemize}
\item There are two vertices for each generator $\gamma$ of $A_{X(\cL)}$ (or alternatively, each
$\gamma \in \cL$), one corresponding to the positive direction and the other to the negative direction.
\item A collection of vertices forms a simplex if they correspond to distinct generators of
$A_{X_{\cL}}$ all of which commute.
\end{itemize}
\end{itemize}
From this description, we see that the fact that $U$ satisfies the injectivity criterion 
ensures that $f$ takes the link of $v$ injectively into a full
subcomplex of the link of $f(v)$ in $\cS(X(\cL))$, as desired. \qed

\section{Bounds on surface groups}
\label{section:surfacegroups}

We now study the lower central series of surface groups and prove
Theorem \ref{maintheorem:surfacegroups}.

\begin{proof}[{Proof of Theorem \ref{maintheorem:surfacegroups}}]
We start by recalling the statement.  For some $g \geq 2$, let $\Sigma_g$ be a closed oriented genus $g$ surface
equipped with a basepoint $\ast$ and let $S = \{a_1,b_1,\ldots,a_g,b_g\}$ be the standard basis for $\pi = \pi_1(\Sigma_g,\ast)$.
Our goal is to prove that
$d_{\pi,S}(k) \geq \frac{1}{4} k$ for all $k \geq 1$.  Equivalently, consider
some nontrivial $w \in \gamma_k(\pi)$.  We must prove that $\|w\|_S \geq \frac{1}{4} k$.

What we will do is find a finite graph $X$ and an injective homomorphism
$\phi\colon \pi \rightarrow A_X$ such that letting $T = V(X)$ be the generating set for $A_X$, we have
$\|\phi(s)\|_{T} \leq 4$ for all $s \in S$.
We then have
$\phi(w) \in \gamma_k(A_X)$, and since $\phi$ is injective we have $\phi(w) \neq 1$.
Since $\pi$ is nonabelian the graph $X$ is not a complete graph, so
we can apply Theorem \ref{maintheorem:raag} to deduce that
$\|\phi(w)\|_{T} \geq k$.  Since $\|\phi(s)\|_{T} \leq 4$ for all $s \in S$,
we conclude that
\[\|w\|_S \geq \frac{1}{4} \|\phi(w)\|_{T} \geq \frac{1}{4} k,\]
as desired.

It remains to construct $X$ and $\phi$.  We can draw the elements of $S$ as follows, where $a_k$ ``encircles'' the $k$th hole from the left:\\
\centerline{\psfig{file=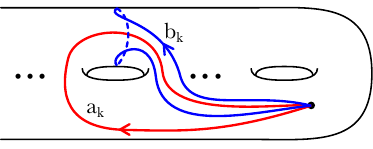,scale=1}}
Let
\[\cL = \{x_0,\ldots,x_g,y_1,\ldots,y_g,z\}\]
be the following simple dissection on $\Sigma_g$:\\
\centerline{\psfig{file=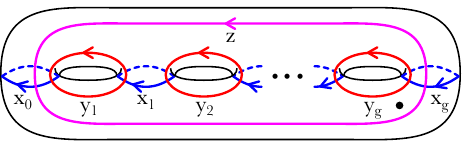,scale=1}}
Let $\phi\colon \pi \rightarrow A_{X(\cL)}$ be the homomorphism obtained by applying the Crisp--Wiest
construction to $\cL$ and let $T = V(X(\cL))$ be the generating set for $A_{X(\cL)}$.
There are four components of $\Sigma_g \setminus G(\cL)$, and by inspection each of them
satisfies the injectivity criterion from \S \ref{section:injectivitycriterion}.  
Proposition \ref{proposition:crispwiest} thus implies that $\phi$ is injective.
By construction, the following hold:
\begin{align*}
\phi(a_k) &= x_{k-1} x_k^{-1}, \\
\phi(b_k) &= x_k z y_k x_k^{-1}.
\end{align*}
These formulas imply that
$\|\phi(s)\|_{T} \leq 4$ for all $s \in S$, as desired.
\end{proof}


\begin{thebibliography}{99}

\bibitem{Baudisch}
A. Baudisch,
Subgroups of semifree groups,
Acta Math. Acad. Sci. Hungar. 38 (1981), no.1--4, 19--28.

\bibitem{Baumslag}
G. Baumslag, On generalised free products, Math. Z. 78 (1962), 423--438. 

\bibitem{CharneySurvey}
R.~M. Charney, An introduction to right-angled Artin groups, Geom. Dedicata 125 (2007), 141--158.  \arxiv{math/0610668}

\bibitem{CrispSageevSapir}
J.~S. Crisp, M. Sageev\ and\ M.~V. Sapir, Surface subgroups of right-angled Artin groups, Internat. J. Algebra Comput. 18 (2008), no.~3, 443--491. \arxiv{0707.1144}

\bibitem{CrispWiest}
J.~S. Crisp\ and\ B. Wiest, Embeddings of graph braid and surface groups in right-angled Artin groups and braid groups, Algebr. Geom. Topol. 4 (2004), 439--472. \arxiv{math/0303217}

\bibitem{DromsThesis}
C. Droms, Graph Groups, PhD thesis, Syracuse University, 1983.

\bibitem{Elkasapy}
A.~I. Elkasapy, A new construction for the shortest non-trivial element in the lower central series, preprint, Oct 2016. \arxiv{1610.09725}

\bibitem{ElkasapyThom}
A.~I. Elkasapy\ and\ A. Thom, On the length of the shortest non-trivial element in the derived and the lower central series, J. Group Theory 18 (2015), no.~5, 793--804. \arxiv{1311.0138}

\bibitem{FoxFree1}
R.~H. Fox, Free differential calculus. I. Derivation in the free group ring, Ann. of Math. (2) 57 (1953), 547--560.

\bibitem{Frederick}
K.~N. Frederick, The Hopfian property for a class of fundamental groups, Comm. Pure Appl. Math. 16 (1963), 1--8. 

\bibitem{GreenThesis}
E. R. Green, Graph products of groups, PhD thesis, University of Leeds, 1990.

\bibitem{HaglundWise}
F. Haglund \& D. Wise,
Special cube complexes,
Geom. Funct. Anal. 17 (2008), no.5, 1551--1620.

\bibitem{KimSurface}
S. Kim, On right-angled Artin groups without surface subgroups, Groups Geom. Dyn. 4 (2010), no.~2, 275--307. \arxiv{0811.1946}

\bibitem{Kuperberg}
G. Kuperberg, Breaking the cubic barrier in the Solovay-Kitaev algorithm, preprint, June 2023. \arxiv{2306.13158}

\bibitem{Magnus}
W. Magnus, Beziehungen zwischen Gruppen und Idealen in einem speziellen Ring, Math. Ann. 111 (1935), no.~1, 259--280. 

\bibitem{MagnusKarrassSolitar}
W. Magnus, A. Karrass\ and\ D.~M. Solitar, {\it Combinatorial group theory}, second revised edition, Dover Publications, Inc., New York, 1976. 

\bibitem{MalesteinPutmanFree}
J. Malestein\ and\ A. Putman, On the self-intersections of curves deep in the lower central series of a surface group, Geom. Dedicata 149 (2010), 73--84. \arxiv{0901.2561}

\bibitem{ScottWallTopological}
G.~P. Scott\ and\ C.~T.~C. Wall, Topological methods in group theory, in {\it Homological group theory (Proc. Sympos., Durham, 1977)}, 137--203, London Math. Soc. Lecture Note Ser., 36, Cambridge Univ. Press, Cambridge. 

\bibitem{ServatiusAutos}
H. Servatius, Automorphisms of graph groups, J. Algebra 126 (1989), no.~1, 34--60.

\bibitem{ServatiusDromsServatius}
H. Servatius, C. Droms\ and\ B. Servatius, Surface subgroups of graph groups, Proc. Amer. Math. Soc. 106 (1989), no.~3, 573--578.

\bibitem{WadeSurvey}
R.~D. Wade, The lower central series of a right-angled Artin group, Enseign. Math. 61 (2015), no.~3-4, 343--371. \arxiv{1109.1722}

\bibitem{WiseBook}
D.~T. Wise, {\it From riches to raags: 3-manifolds, right-angled Artin groups, and cubical geometry}, CBMS Regional Conference Series in Mathematics, 117, Published for the Conference Board of the Mathematical Sciences, Washington, DC, 2012.

\end{thebibliography}
\end{document}